\documentclass[11pt]{article}

\usepackage{enumerate}
\usepackage{amsthm,amsmath,amssymb}
\usepackage{graphicx}
\usepackage{lineno}
\usepackage[colorlinks=true,citecolor=black,linkcolor=black,urlcolor=blue]{hyperref}
\usepackage[english]{babel}
\usepackage{amsfonts}
\usepackage{epsfig, subfigure}
\usepackage{amscd,latexsym}
\usepackage{comment}
\usepackage{url}
\usepackage{color}
\usepackage{authblk}
\usepackage{subfigure}

\theoremstyle{plain}
\newtheorem{theorem}{Theorem}
\newtheorem{lemma}[theorem]{Lemma}
\newtheorem{cor}[theorem]{Corollary}
\newtheorem{prop}[theorem]{Proposition}

{\itshape}{\rmfamily}
\newtheorem{remark}[theorem]{Remark}

\theoremstyle{plain}

\begin{document}
	
	\title{
		Graphs with isolation number equal to one third of the order
	}
	
		\author[1]{Magdalena Lema\'nska \thanks{magleman@pg.edu.pl}}

   \author[2]{Merc\`e Mora\thanks{Partially supported by projects 
			PID2019-104129GB-I00/MCIN/AEI/10.13039/501100011033
			of the Spanish Ministry of Science and Innovation
			and Gen.Cat. DGR2021SGR00266, merce.mora@upc.edu}}
	
	\author[3]{Mar\'ia Jos\'e Souto-Salorio\thanks{Partially supported by project PID2020-113230RB-C21 of the Spanish Ministry of Science and Innovation, maria.souto.salorio@udc.es}}
	
   \affil[1]{Institute of Mathematics, Gda\'nsk University of Technology\\ Poland}
	
	\affil[2]{Departament de Matem\`atiques, Universitat Polit\`ecnica de Catalunya, Spain}
	
	\affil[3]{Departamento de Ciencias da Computaci\'on e Tecnolox\'ias da Información, Universidade da Coru\~{n}a\\ Spain}
 \date{}

\maketitle

\begin{abstract}
A set $D$ of vertices of a graph $G$ is isolating if the set of vertices not in $D$ and with no neighbor in $D$ is independent. The isolation number of $G$, denoted by $\iota (G)$, is the minimum cardinality of an isolating set of $G$. It is known that $\iota (G)\le n/3$, if $G$ is a connected graph of order $n$, $n\ge 3$, distinct from $C_5$. 
The main result of this work is the characterisation of unicyclic and block graphs
of order $n$ with isolating number equal to $n/3$.
Moreover, we provide a family of general graphs attaining this upper bound on the isolation number.
\end{abstract}

\section{Introduction}
\label{sectIntro}
Domination in graphs has deserved a lot of attention since it was introduced in the fifties motivated by chessboard problems, among others. A set $D$ of vertices of a graph $G$ is {\it dominating} if every vertex not in $D$ has at least one neighbor in $D$. The {\it domination number} of a graph $G$, denoted by $\gamma (G)$, is the minimum cardinality of a dominating set.
There is  an extensive literature on dominating sets in graphs, see for example the book \cite{haynes} and references therein. 
The  definition of dominating set can be reformulated as follows.
For every graph $G=(V,E)$, let   $N_G[v] = \{v\}\cup   \{u \in V :  \  uv \in E \}$ be the {\it closed neighborhood} of the vertex $v\in V$. 
If $S\subseteq V$, then $N_G[S]=\cup_{v\in S} N_G[v]$.
If the graph $G$ is clear from context, we simply write $N[u]$ and $N[S]$, instead of $N_G[u]$ and $N_G[S]$.
With this terminology, $D$ is a dominating set of $G$  if and only if  $V=N[D]$.
The concept of isolation arises by relaxing this condition. 

The requirement $N[D]=V$ is important in a wide kind of applications  in real-life scenarios.  Nevertheless,  sometimes this condition is not  necessary  or profitable,  and then we  could deal with a relaxed condition. The notion of isolating set is an example of this.  
In \cite{adriana},
  the authors introduce the concept of isolation in a graph. Concretely, for a family $\mathcal{F}$  of graphs, a set $S$ of vertices of a graph $G$ is $\mathcal{F}$-{\it isolating}  if the graph induced by the set $V- N_G[S]$ contains no member of $\mathcal{F}$ as a subgraph. 
In particular,  $\{K_1\}$-isolating sets are the usual  dominating sets. 
The vertices not dominated by a $ \{ K_2 \}$-isolating set form an independent set. 
We   use for short   {\it isolating set } instead of $\{K_2\}$-isolating set.

Hence, a set $D$ of vertices of a graph $G$ is {\it isolating} if the set of vertices not in $N[D]$ is independent \cite{adriana}. The {\it isolation number} of $G$, denoted by $\iota(G)$, is the minimum cardinality of an isolating set.
The following upper bound for the isolation number has been proven.
\begin{theorem}\label{thm:caro} \cite{adriana} Let $G$ be a connected graph on $n\ge 3$ vertices different from $C_5$. Then $\iota (G)\le n/3$ and this bound is sharp.
\end{theorem} 

Some other upper bounds for $\{K_k\}$-isolation have been proved in \cite{borg20, borg22}, for $k\ge 1$. 
Concretely, a bound in terms of the order and $k$ is given in \cite{borg20}, and a bound in terms of the number of edges and $k$ is given in \cite{borg22}. Also, all graphs attaining this last upper bound are characterized. 
Our goal is to characterize all connected graphs of order $n$ such that $\iota (G)=n/3$ (Problem 3.2 posed in \cite{borg20}, for $k=2$). This problem is already solved for trees~\cite{my}. In this work, we give a family of graphs attaining the upper bound, that includes the family given in \cite{borg20} to prove that this bound is tight, and solve it completely for unicyclic graphs  and block graphs. The techniques used to prove the characterization for unicyclic and block graphs are quite different. 

The paper is organised as follows. 
Section~\ref{sec:preliminaries} is devoted to introduce some terminology and preliminary results.
Unicyclic and block graphs attaining the upper bound on the isolation number are characterized in  Section~\ref{sec:unicyclic} and 
Section~\ref{sec:block}, respectively.
We finish with some concluding remarks in Section~\ref{sec:concluding}.

\section{Preliminaries}\label{sec:preliminaries}

A \emph{leaf} of a graph $G$ is a vertex of degree 1 and a \emph{support vertex} is a vertex adjacent to a leaf.
Let $\mathcal{T}$ be the family of trees $T$, described in \cite{my}, that can be obtained from a sequence of trees $T_1, \ldots, T_j$,  $j\ge 1$,
such that $T_1$ is a path $P_3$; $T = T_j$;  and, if $1\le i \leq j-1$, then $T_{i+1}$ can be obtained from $T_i$ by
adding a path $P_3$ and an edge $xy$, where $x$ is a vertex at a distance two from a leaf of $T_i$ and $y$ is a leaf of the path $P_3.$
\begin{theorem}\label{thm:ls} \cite{my} If $T$ is a tree of order $n$, then $\iota (G)= n/3$ if and only if  $T \in \mathcal{T}.$
\end{theorem} 

Observe that the family $\mathcal{T}$ can also be described as follows. 
A tree $T$ belongs to $\mathcal{T}$ if and only if it can be obtained by attaching exactly one copy of a path $P_3$ at every vertex $v$ of a tree $T_0$ by identifying $v$ with a leaf of $P_3$ (see an example in Figure~\ref{fig:familyT}). Thus, if $T$ is a tree of order $n$ belonging to $\mathcal{T}$, then $n$ is a multiple of 3 and, for $n>3$, one third of the vertices of $T$ are leaves and one third are support vertices. 
\begin{figure}[ht]
\begin{center}
	\includegraphics[width=0.25\textwidth]{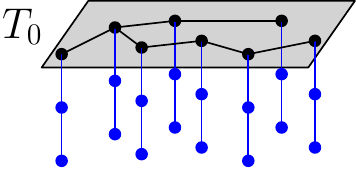}
	\caption{A tree $T$ of the family $\mathcal{T}$. The tree $T_0$ is the tree drawn in the gray region.}	
	\label{fig:familyT}
\end{center}
\end{figure}

For every tree $T\in \mathcal{T}$ of order at least 6, 
we denote by $L(T)$, $S(T)$ and $A(T)$ the set of leaves, support vertices and the remaining vertices of $T$, respectively. If $T$ has order $3$, then $L(T)$ will be a set with exactly one of its leaves and $A(T)$ the set containing the other leaf. If the tree $T$ is clear from context, we write simply $L$, $S$ and $A$. 
A 3-\emph{set} of $T$  is a set of cardinality 3 containing a leaf, its support vertex and the vertex in $A$ nearest to the leaf.
If $u\in V(T)$, we denote by $P_3(u)$ the 3-set containing $u$, and by $a(u)$, $s(u)$ and $\ell (u)$ the vertex in $A$, the support vertex and the leaf of $P_3(u)$, respectively.

Notice that every isolating set $D$ of $T$ contains at least one vertex of each attached copy of $P_3$, otherwise the two vertices of $P_3$ not in $T_0$ induce a copy of $K_2$ in $V(T)\setminus N[D]$. Moreover, the set of all vertices of $T_0$ and the set of all support vertices of $T$ are examples of minimum isolating sets of $T$.

This result reminds us the characterization of graphs attaining the upper bound on the domination number. It is known that $\gamma (G)\le n/2$, for every graph $G$ of order $n$ with no isolated vertices \cite{ore}. Moreover, all graphs attaining this upper bound have been characterized \cite{Fink,payan} and can be described as the graphs obtained by attaching a copy of $K_2$ at every vertex of a given graph.

These kind of graphs suggested us the following construction, that provides a family of connected graphs attaining the upper bound on the isolation number. Let $P_3$, $C_3$, $H_6^1$, $H_6^{2a}$, $H_6^{2b}$ and $H_6^{3}$ denote the graphs depicted in Figure~\ref{fig:familyTG}, left.
The family $\mathcal{G}$ consists of all graphs obtained by attaching exactly one copy of one of the graphs $P_3$, $C_3$, $H_6^1$, $H_6^{2a}$, $H_6^{2b}$ or $H_6^{3}$ at every vertex $v$ of a connected graph $G_0$, by identifying $v$ with the circled vertex of the attached graph (see an example in Figure~\ref{fig:familyTG}, right). 
Observe that the trees in $\mathcal{G}$ are precisely the trees in the family $\mathcal{T}$.

Notice that every minimum isolating set of a graph belonging to $\mathcal{G}$  contains at least one vertex of each attached graph, if it is $P_3$ or $C_3$, and at least 2 vertices of each attached graph in the remaining cases. 
Indeed, if there is no vertex from a copy of $P_3$ or $C_3$ in an isolating set $D$, 
then the two vertices of $P_3$ or $C_3$ not in $G_0$ induce a copy of $K_2$ in $V(G)\setminus N[D]$, and  
if there is at most one vertex of a copy of $H_6^1$, $H_6^{2a}$, $H_6^{2b}$ or $H_6^{3}$
in an isolating set $D$, then there are always two vertices of the cycle of order 5 formed by the vertices not belonging to $G_0$ that induce a copy of $K_2$ in $V(G)\setminus N[D]$. 
Therefore, the following result holds. 
\begin{figure}[ht]
\begin{center}
	\includegraphics[width=0.8\textwidth]{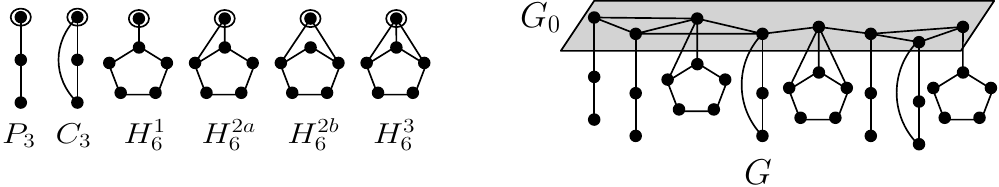}
	\caption{Left, the graphs $P_3$, $C_3$, $H_6^1$, $H_6^{2a}$, $H_6^{2b}$ and $H_6^{3}$. Right, a graph $G\in\mathcal{G}$.}	
	\label{fig:familyTG}
\end{center}
\end{figure}
\begin{prop}\label{prop:tipo}  If $G$ is a graph of order $n$ that belongs to $\mathcal{G}$,  then $\iota (G)= n/3$.
\end{prop} 
%

\section{Unicyclic graphs}\label{sec:unicyclic}

A connected graph $G$ is {\it unicyclic} if it contains exactly one cycle. 
Let $\mathcal{U}$ be the family of unicyclic graphs belonging to $\mathcal{G}$. Notice that these graphs are obtained 
by either attaching a copy of $P_3$ at every vertex of a unicyclic graph;
or a copy of $C_3$ at a vertex of a tree $T_0$ and a copy of $P_3$ at any other vertex of $T_0$;
or a copy of $H_6^1$ at a vertex of a tree $T_0$ and a copy of $P_3$ at any other vertex of $T_0$
(see some examples in Figure~\ref{fig:familyUnew}). 
\begin{figure}[ht!]
	\centering
	\includegraphics[width=0.70\textwidth]{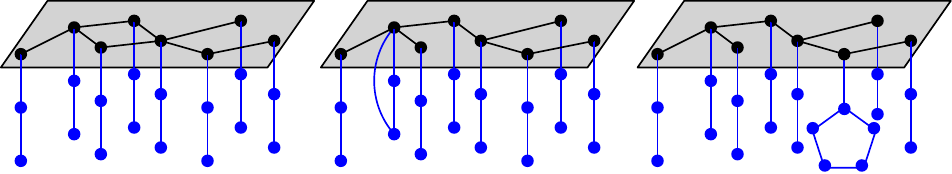}
	\caption{The family $\mathcal{U}$ consists of all unicyclic graphs of the family $\mathcal{G}$.}
	\label{fig:familyUnew}
\end{figure}

We will see that the only unicyclic graphs  not in $\mathcal{U}$ attaining the upper bound on the isolation number are the cycles of order $6$ and $9$.
The proof is based on the result for trees~\cite{my}. 
We begin with some terminology and technical lemmas that will be used in the proof.

\begin{lemma}\label{rem:isolatingleaf}
	If a graph $G$ has a leaf $u$ such that its support vertex $v$ has degree $2$, then there exists a minimum isolating set of $G$ containing neither $u$ nor $v$.
\end{lemma}
\begin{proof}
Let $w$ be the vertex adjacent to $v$ and distinct from $u$. 
If $D$ is a minimum isolating set, then no both vertices $u$ and $v$ belong to $D$, since $D\setminus \{u\}$ would be an isolating set of cardinality lower that $\iota (G)$, a contradiction.
Let $D$ be a minimum isolating set of $G$. If neither $u$, nor $v$ belong to $D$, then $D$ satisfies the desired condition.
If $u\in D$, then $(D\setminus \{u\})\cup \{w\}$ is also a minimum isolating set satisfying the desired condition.
Similarly, if $D$ contains $v$, then $(D\setminus \{v\})\cup \{w\}$ is also a minimum isolating set satisfying the desired condition.
\end{proof}

Let $G$ be a unicyclic graph. For every vertex $u$ of the only cycle $C$ of $G$, we denote by $T(u)$ the component containing $u$ after removing from $G$ the edges of $C$ incident with $u$. Notice that $T(u)$ is a tree. We refer to it as the {\em hanging-tree} of $u$ in $G$.

\begin{lemma}\label{lem:T1T2} If $G$ is a unicyclic graph of order a multiple of 3, then there exists a set $X$ of at most three consecutive vertices in the only cycle $C$ of $G$ such that the cardinality of the set $\cup_{u\in X}V(T(u))$ is a multiple of 3.
\end{lemma}
\begin{proof} 
If there is a vertex $u\in V(C)$ such that the order of $T(u)$ is a multiple of 3, then $X=\{u\}$.
	If $T(u)$ has order 1 modulo 3 for every $u\in V(C)$ or $T(u)$ has order 2 modulo 3 for every  $u\in V(C)$, then $X$ is a set of three consecutive vertices of $C$. 
	If the preceding cases do not hold, then there are two adjacent vertices $u$ and $v$ in $C$ such that one of the orders of $T(u)$ and $T(v)$ is 1 modulo 3 and the other is 2 modulo 3. Hence, the statement holds for $X=\{u,v\}$.
\end{proof}

 We are going to use the following obvious remark in many cases of the proof of the main theorem of this section.
\begin{remark}\label{rem:edge}
If $e$ is an edge of a graph $G$ and $D$ is an isolating set of $G-e$ such that at least one end-vertex of $e$ belongs to $N[D]$, then $D$ is an isolating set of $G$.
\end{remark}

\begin{theorem} If $G$ is a unicyclic graph of order $n$, then $\iota (G)=n/3$ if and only if $G\in \mathcal{U}\cup \{C_6,C_9\}$.
\end{theorem}

\begin{proof} Obviously, $\iota (C_6)=2$ and $\iota (C_9)=3$ and, from Proposition~\ref{prop:tipo}, we know that every graph of order $n$ belonging to $\mathcal{U}$ satisfies $\iota (G)=n/3$.

Next we prove that any unicyclic graph  $G$ of order $n$ not in $\mathcal{U}\cup \{C_3,C_9\}$ satisfies $i(G)<n/3$.
If $n$ is not a multiple of 3, then $\iota (G)<n/3$.
Suppose now that $G$ is a unicyclic graph not in $\mathcal{U}\cup \{C_3,C_9\}$ such that $n$ is a multiple of 3.
Let $C$ be the only cycle of $G$. 
To proceed with the proof, we seek for two edges $a$ and $b$ of the cycle $C$ such that $G-\{a,b\}$ is a forest consisting of two trees $T_1$ and $T_2$ of order a multiple of 3, so that we can apply the already known results for trees to $T_1$ and $T_2$.

Notice that it is possible to find such edges if there exists a set  $X$, $X\not= V(C)$, satisfying the conditions of Lemma~\ref{lem:T1T2}. In this case, if $T_1$ is the graph induced by the set of vertices $\cup_{u\in X}V(T(u))$ and  $T_2$ is the graph induced by the remaining vertices of $G$, 
then $T_1$ and $T_2$ are trees of order a multiple of 3 and $G-\{a,b\}=T_1\cup T_2$, where $a$ and $b$ are the edges of $C$ with an end-vertex in $X$ and the other end-vertex in $V(C)\setminus X$.

Observe that there exists such a set $X$ with $X\not= V(C)$, except when $C$ has order $3$ and the three hanging-trees have order 1 modulo 3 or the three hanging-trees have order 2 modulo 3. 
Hence, to proceed with the proof, we distinguish the cases (i) $X\not= V(C)$, for some set $X$ satisfying the conditions of Lemma~\ref{lem:T1T2}
and (ii) $X= V(C)$, for every set $X$ satisfying the conditions of Lemma~\ref{lem:T1T2}.

\medskip
	
	\noindent
	{\bf Case (i):}  $X\not= V(C)$, for some set $X$ satisfying the conditions of Lemma~\ref{lem:T1T2}. Then $G-\{a,b\}$ is a forest with two trees $T_1$ and $T_2$ of order a multiple of 3, for some edges $a$ and $b$ of $C$.
 Let $W_i$ be a minimum isolating set of $T_i$, $i\in \{1,2\}$. If $T_i\in \mathcal{T}$, we take $W_i$ to be also a dominating set, that is always possible (see \cite{my}).
 We consider the following cases depending on the trees $T_1$, $T_2$ belong or not to $\mathcal{T}$.
	\begin{enumerate}[$\bullet$]
		\item  If $T_1,T_2\notin \mathcal{T},$ then $W_1\cup W_2\cup  \{u \}$ is an isolating set of $G$, where $u$ is one of the vertices of $T_1$ belonging to $C$ (the one in the middle, if $|V(T_1)\cap V(C)|=3$).
		Thus,  
		$$\iota (G)\le \frac{n_1}3 -1 +\frac{n_2}3 -1+1=\frac n3 -1<\frac n3.$$
		\item If  $T_1\notin \mathcal{T}$ and $T_2\in \mathcal{T}$, then  $W_1\cup W_2$ is an isolating set of $G$, because $W_2$ is a dominating set of $T_2$.  Thus, 
		$$\iota (G)\le \frac{n_1}3 -1+\frac{n_2}3 =\frac n3 -1<\frac n3.$$
		
		\item If  $T_1\in \mathcal{T}$ and $T_2\notin \mathcal{T}$, then  we deduce $\iota (G)<n/3$ arguing as in the previous case.
		
		\item If $T_1,T_2\in \mathcal{T}$, let $a=u_1u_2$ and $b=v_1v_2$ be the two edges of $C$ with an endpoint in $T_1$ and the other in $T_2$, where $u_1,v_1\in V(T_1)$ and $u_2,v_2\in V(T_2)$.
		By construction, the distance between the vertices $u_1$ and $v_1$ is at most 2.
		Recall that $G-\{a,b\}$ is a forest with components $T_1$ and $T_2$.
		Moreover, if $A=A(T_1)\cup A(T_2)$, $S=S(T_1)\cup S(T_2)$ and $L=L(T_1)\cup L(T_2)$, then $|A|=|S|=|L|=n/3$.
Since $G-\{a,b\})=T_1\cup T_2$, for every $u\in V(G)$, we refer to $P_3(u)$ as the 3-set of $u$ in the tree $T_1$ or $T_2$ containing $u$, so that  $\ell(u)$, $s(u)$ and $a(u)$ denote the leaf, support vertex and the remaining vertex in $P_3(u)$, respectively.

		\emph{Case 1:} $T_1$ and $T_2$ have order at least 6.
  Observe that in this case every vertex in  $A(T_1)$ (resp. $A(T_2)$) has a neighbor in $A(T_1)$  (resp. $A(T_2)$). 
		Since $G\notin \mathcal{U}$, at least one of the edges $a$ or $b$, say $a$, has an end-vertex not in $A$. 
  Hence,  it is enough to consider the following cases.
		\begin{enumerate}[-]
			\item \emph{Case 1.1:} $u_1,u_2\notin A$.  The set $D=(A\setminus \{a(u_1),a(u_2)\})\cup \{u_1\}$ is isolating in $G-b$. Notice that $v_1\in A(T_1)\cup P_3(u_1)$, since $d(u_1,v_1)\le 2$.
   Therefore, $v_1\in N[D]$ and  $D$ is isolating in $G$, by Remark~\ref{rem:edge}.  
			\item  \emph{Case 1.2:} $u_1\in A$ and $u_2\notin A$. { If $v_1,v_2\notin A$, then we proceed an in case 1.1. Suppose now that $v_1\in A$ or $v_2\in A$.}  The set $D=A\setminus \{a(u_2)\}$ is isolating in $G-b$.  
   If $v_1\not= \ell(u_1)$, then $v_1\in N[D]$, because $d(u_1,v_1)\le 2$.
   If $v_1= \ell(u_1)$, then $v_2\in A$, so that $v_2\in N[D]$.  
   Hence, in all cases we obtain an isolating set $D$ of $G$, by Remark~\ref{rem:edge}.
			
			\item  \emph{Case 1.3:} $u_1\notin A$ and $u_2\in A$. The set $D=A\setminus \{a(u_1)\}$ is an isolating set for $G-b$. 
   If $v_1$ or $v_2$ is in $N[D]$, then $D$ is isolating in $G$ by Remark~\ref{rem:edge}. Otherwise, 
   $v_2\in L$ and $v_1$ is the only vertex in $\{s(u_1),\ell(u_1) \}$ different from $u_1$. In such a case, $v_1,v_2\notin A$, {and we proceed as in Case 1.1.}		\end{enumerate}
		Therefore, it is possible to give an isolating set of cardinality $n/3-1$ in all cases, implying that $\iota (G)< n/3$.
		
		\noindent 
		\emph{Case 2:} one of the trees $T_1$ or $T_2$ has order 3 and the other has order at least 6. We may assume that $T_1$ is $P_3$. 
		We distinguish the following cases according to the number of vertices of $T_1$ belonging to the cycle $C$.
		\begin{enumerate}[-]
			\item \emph{Case 2.1:} $|V(C)\cap V(T_1)|=1$. Then $u_1=v_1$. 
If $u_1$ is the vertex of degree 2 in $T_1(=P_3)$, then $(A(T_2)\setminus\{a(u_2)\})\cup \{u_2\}$ is isolating in $G$.
  Otherwise, $u_1$ is a leaf of $T_1(=P_3)$. In such a case, $u_2$ or $v_2$, say $u_2$, does not belong to $A(T_2)$,  because $G\notin \mathcal{U}$.
			Hence, the set
			$D=(A(T_2)\setminus\{a(u_2)\})\cup \{u_1\}$ is isolating in $G-b$, and $v_1=u_1\in N[D]$. Therefore, $D$  is isolating in $G$, by Remark~\ref{rem:edge}.

			\item \emph{Case 2.2:} $|V(C)\cap V(T_1)|=2$. We may assume that $u_1$ is a leaf and $v_1$ is the support vertex in $T_1$. Suppose first that $a(u_2)\not= a(v_2)$.  If $v_2$ is in $A(T_2)$, then $(A(T_2)\setminus\{ a(u_2)\})\cup \{u_2\}$ is isolating in $G$.   {If $v_2\notin A(T_2)$, then  $(A(T_2)\setminus\{a(v_2)\})\cup \{v_1\}$ is isolating in $G$.
  Now suppose that $a(u_2)= a(v_2)$. If $u_2=v_2$, then $(A(T_2)\setminus\{ a(u_2)\})\cup \{u_2\}$ is isolating in $G$. If $u_2\not= v_2$ we distinguish some cases.
  If $u_2$ and is a leaf, then  $(A(T_2)\setminus\{ a(u_2)\})\cup \{u_1\}$ is isolating in $G$.
  If $v_2$ is a leaf, then  $(A(T_2)\setminus\{ a(v_2)\})\cup \{v_1\}$ is isolating in $G$. 
  If $u_2\in A(T_2)$ and $v_2\in S(T_2)$, then $(A(T_2)\setminus\{ a(v_2)\})\cup \{v_1\}$
  is isolating in $G$. 
  If $u_2\in S(T_2)$ and $v_2\in A(T_2)$, then $(A(T_2)\setminus\{ a(u_2)\})\cup \{u_1\}$
  is isolating in $G$. }
   
			\item \emph{Case 2.3:} $|V(C)\cap V(T_1)|=3$. Then $u_1$ and $v_1$ are the leaves of $T_1$. 
			 If $a(u_2)\not= a(v_2)$, let $D=(A(T_2)\setminus \{a(u_2), a(v_2)\})\cup \{u_2,v_2\}$. Observe that $D$ is an isolating set of $G$, except when $T_2$ has order 6 and $u_2$ and $v_2$ are leaves of $T_2$, that is, whenever $G$ is  a cycle of order 9, but in this case, $G\in \mathcal{U}\cup \{C_6,C_9\}$, a contradiction. 
            If $a(u_2)= a(v_2)$ and $u_2=v_2$, then 
$D=(A(T_2)\setminus \{a(u_2)\})\cup \{u_2\}$ is an isolating set of $G$.  
If $a(u_2)= a(v_2)$ and $u_2\not= v_2$, then $\{u_2,v_2\}\not= \{s(u_2),\ell (u_2)\}$, because otherwise $G\in \mathcal{U}$. 
Hence, $u_2$ or $v_2$ is $a(u_2)$. Assume that $u_2=a(u_2)$.
Then $(A(T_2)\setminus \{ a(u_2) \}\cup \{ v_1\}$ 
 is an isolating set of $G$.
		\end{enumerate}
		Therefore,  it is possible to give an isolating set of $G$ with at most $n/3-1$ vertices in all cases. Hence, $\iota (G)< n/3$.
  
		\begin{figure}[ht!]
			\centering
			\includegraphics[width=0.7\textwidth]{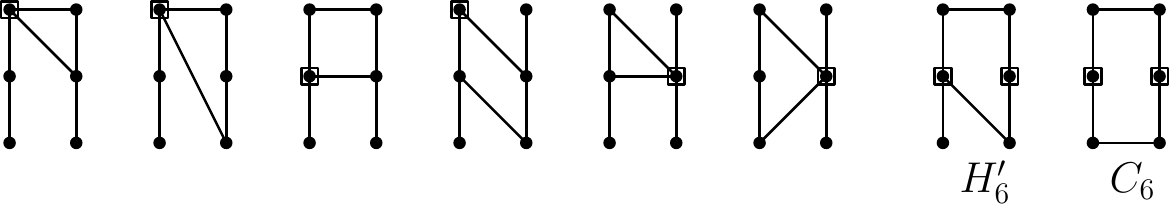}
			\caption{Unicyclic graphs of order $6$ such that the components of the forest obtained after removing two suitable edges of the cycle are two paths of order 3. Squared vertices form a minimum isolating set. The only graphs with isolation number equal to 2 are $H_6'$ and $C_6$, that belong to $\mathcal{U}\cup \{C_6,C_9\}$.}
			\label{fig:case33}
		\end{figure}	
  
		\noindent
		\emph{Case 3:} $T_1$ and $T_2$ have order 3. In such a case, by adding two edges between $T_1$ and $T_2$ in all possible ways, we obtain that $G$ is one of the graphs depicted in Figure~\ref{fig:case33}. All these graphs have isolation number equal to $1$, except $C_6$ and $H_6'$, that have isolation number $2$, but $C_6,H_6'\in \mathcal{U}$. Hence  $\iota (G)<n/3$, whenever $G\notin \mathcal{U}\cup \{C_6,C_9\}$.
	\end{enumerate}	
	
	\noindent
	{\bf Case (ii):} $X= V(C)$, for every set $X$ satisfying the conditions of Lemma~\ref{lem:T1T2}. Then, $C$ is a cycle of order $3$ and the three hanging-trees have order 2 modulo 3 or the three hanging-trees have order 1 modulo 3. Let $V(C)=\{x_1,x_2,x_3\}$ and let $W_i$ be a minimum isolating set of $T(x_i)$, for every $i\in \{1,2,3\}$. Since $T(x_i)\notin \mathcal{T}$, we have  $|W_i|< |V(T(x_i))|/3.$
	\medskip
	
	\begin{enumerate}[$\bullet$]
		\item 
		If  $T(x_i)$ has order 2 modulo 3 for every $i\in \{ 1,2,3\}$,  then  $|V(T(x_i))|=n_i'=3 k_i+2$ and $n/3=k_1+k_2+k_3+2$, for some integers $k_1,k_2,k_3\ge 0$. By symmetry, it is enough to consider the following cases. 
		\begin{enumerate}[-]
			\item 
			If $k_1,k_2,k_3\ge 1$, then  $W_1\cup W_2\cup W_3\cup \{x_1\}$ is an isolating set of $G$ and $|W_i|\le k_i$, for every $i\in \{1,2,3\}$. Hence,
			$\iota (G)\le k_1+k_2+k_3+1<\frac n3$.
			\item 	If $k_1,k_2\ge1$ and $k_3=0$, then 
			$W_1\cup W_2\cup \{x_3\}$ is an isolating set of $G$. 
			Hence, $$\iota (G)\le k_1+k_2+1<\frac{(3k_1+2) +(3k_2+2)+2}{3}=\frac n3.$$
			\item If $k_1\ge 1$ and $k_2=k_2=0$, then $W_1\cup \{x_3\}$ is an isolating set of $G$.
			Hence, $$\iota (G)\le k_1+1<\frac{(3k_1+2) +2+2}{3}=\frac n3.$$
		\end{enumerate}
		
		\item If $T(x_i)$ has order 1 modulo 3, for every $i\in  \{1,2,3\}$, then $|V(T(x_i))|=n_i'=3 k_i+1$ for some integers $k_1,k_2,k_3\ge 0$, and $n/3=k_1+k_2+k_3+1$. 
		Let
		$T'(x_i)$ be the tree obtained from $T(x_i)$ by adding the two pendant edges $x_ix_j$, $j\not=i$ (see Figure~\ref{fig:caseC3}) Then, $T'(x_i)$ is a  tree of order $3k_i+3$.
  \begin{figure}[ht!]
			\centering
			\includegraphics[width=0.75\linewidth]{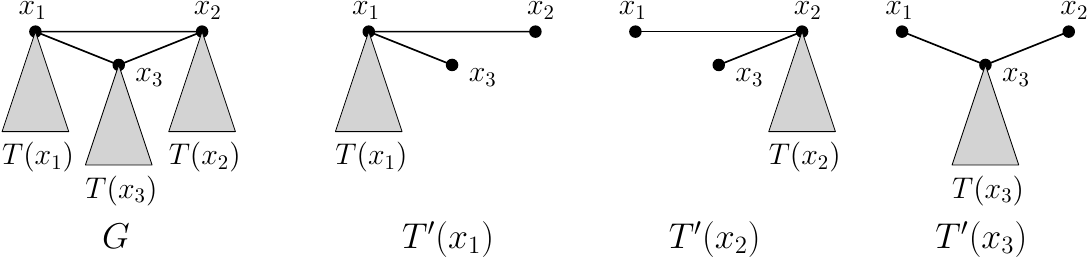}
			\caption{The graph $G$ and the trees $T'(x_1)$, $T'(x_2)$ and $T'(x_3)$, when $C$ is a cycle of order 3 and $T(x_1)$, $T(x_2)$ and $T(x_3)$ are trees of order 1 modulo 3.}
			\label{fig:caseC3}
		\end{figure}
		
		Let $W'_i$ be a minimum  isolating set of $T'(x_i)$ not containing the leaves $x_j$, $j\not= i$, that exists by Lemma~\ref{rem:isolatingleaf}, so that $x_i\in N[W_i]$.
		If $k_i\ge 1$, then  $T(x_i)\notin \mathcal{T}$, because  there are two leaves with the same support vertex and $T'(x_i)$ has order at least 6, implying that $|W_i'|\le k_i$ (see \cite{my}).
  Now we distinguish the following cases.
		\begin{enumerate}[-]
			\item If $k_1=k_2=k_3=0$, then $G$ is a cycle of order 3, so that $G\in \mathcal{U}$.
			
			\item If $k_i\ge 1$ for $i\in \{1,2,3\}$, then 
			$D=W_1'\cup W_2'\cup W_3'$ is an isolating set of $G$ satisfying
			$$|D|=|W_1'|+|W_2'|+|W_3'|\le k_1+k_2+k_3<\frac n3.$$
			
			\item If exactly one of the values $k_1,k_2,k_3$ is equal to 0, assume that $k_3=0$. 
			By construction, $D=W_1'\cup W_2'$ is an isolating set of $G$ satisfying
			$$|D|=|W_1'|+|W_2'|\le k_1+k_2=k_1+k_2+k_3<\frac n3.$$
			
			\item If exactly two of the values $k_1,k_2,k_3$ are equal to 0, assume that $k_1\not=0$. 
			If $G\notin \mathcal{U}$, then $G-x_2x_3$ is a tree $T$ not belonging to $ \mathcal{T}$.
			Consider a minimum isolating set $D$ of $T$ such that $x_1\in D$ and not containing neither $x_2$ nor $x_3$, that exists by Lemma~\ref{rem:isolatingleaf}.
			Then, $D$ is also an isolating set of $G$ such that 
			$|D|<\frac n3$, because $T\notin \mathcal{T}$.
		\end{enumerate}
		Therefore, $\iota (G)<n/3$, whenever $G\notin \mathcal{U}\cup \{C_6,C_9\}$
	\end{enumerate}
\end{proof}

\section{Block graphs}\label{sec:block}

A vertex $v$ of a connected graph $G$ is  
a {\it cut vertex} if the removal of $v$ from $G$ results in a disconnected graph. A {\it block} of a graph $G$ is a maximal connected subgraph of $G$ without cut vertices. A connected graph $G$ is a {\it block graph} if every block of $G$ is a complete graph. 
Let $\mathcal{B}$ denote the family of block graphs belonging to $\mathcal{G}$. Hence,  if $G\in \mathcal{B}$, then $G$ is obtained by attaching $P_3$ or $C_3$ at every vertex of a block graph $G_0$ (see Figure~\ref{fig:familyB}) and the following remark holds. 
\begin{remark}\label{rem:deg3in block} If  $G\in \mathcal{B}$ has order $n$, then $n$ is a multiple of 3 and $\iota (G)=n/3$. Moreover, if $A(G)$ is the set of vertices belonging to $G_0$, then $A(G)$ is a minimum isolating set of $G$ and every vertex of degree at least 3 in $G$ belongs to $A(G)$.   
\end{remark}
\begin{figure}[ht!]
	\centering
	\includegraphics[width=0.65\textwidth]{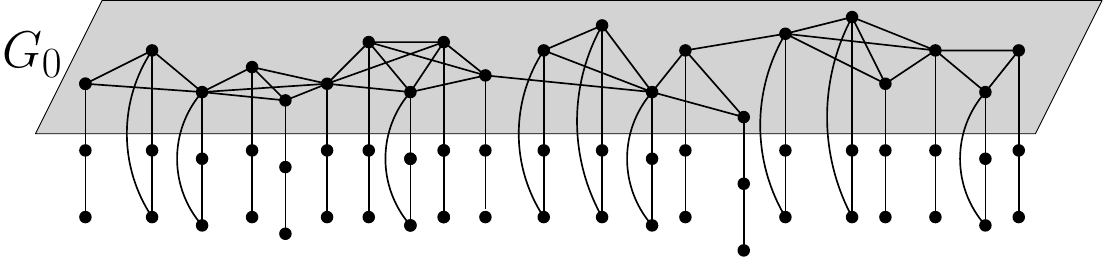}
	\caption{A graph of the family $\mathcal{B}$. The graph $G_0$ is drawn in the gray region.}
	\label{fig:familyB}
\end{figure}
We will prove that the graphs from $\mathcal{B}$ are precisely the block-graphs attaining the upper bound on the isolation number. In order to prove this result, we need first some previous terminology and results.

Let $G=(V,E)$ be a connected graph. A vertex $v\in V$ is  {\it simplicial} if $N_G[v]$ is a complete graph. Notice that if $G$ is a block graph, then  every vertex of $G$ is a simplicial or a cut vertex. 
A block of a block graph $G$ is an {\it end-block} if it contains at most one cut vertex of $G$.
It is well-known that every block graph different from a complete graph has at least two end-blocks~\cite{LYD}. 

Observe that complete graphs are the only block-graphs with no cut vertices and isolation number equal to 1. Hence, $K_3$ is the only complete graph $G$ such that $\iota (G)=\vert V(G)\vert /3$. We begin with some results useful to prove the main theorem of this section.

\begin{prop} \label{propnosimplicial}
    If $G$ is a block graph different from a complete graph, then there is a minimum isolating set containing no simplicial vertices.
\end{prop}
\begin{proof} 
Suppose that $D$ is a minimum isolating set of $G$. 
If $D$ contains no simplicial vertex, then we are done. 
Now suppose that $u\in D$ is a simplicial vertex of a block $B$. 
If there is a cut vertex $v\in D\cap B$, then $N[u]=B\subseteq N[v]$, so that $D-\{u\}$ is also an isolating set of $G$, a contradiction.
Otherwise, for any cut vertex $v\in B$, the set $(D-\{u\})\cup \{ v \}$ is also a minimum isolating set.
Proceeding in this way for every simplicial vertex of $D$, we obtain a minimum isolating set of $G$ containing no simplicial vertex.
\end{proof}

\begin{prop} \label{propblock1}
    Let $G$ be a block graph of order $n$ different from a complete graph. If $\iota(G) = n/3,$ then every block $B$ of $G$ has at most two simplicial vertices.
\end{prop}

\begin{proof} 
Suppose that there is a block $B$ of $G$ containing three or more simplicial vertices.
If $u$ is a simplicial vertex of $B$, then $G-u$ is a block graph and $B-u$ is a block of $G-u$ with at least two simplicial vertices. 
Hence, a set with no simplicial vertices is isolating in $G-u$ if and only if is isolating in $G$.
Since, by Proposition~\ref{propnosimplicial}, every block graph has at least one minimum isolating set with no simplicial vertices, we deduce that $\iota (G)=\iota (G-u)$.
Therefore, $D$ is a minimum isolating set of $G-u$ and we obtain $n/3 = \iota(G)  = \iota(G-u)\leq  (n-1)/3,$ which is a contradiction.
\end{proof}

\begin{cor} \label{propblock2}
If $G$ is a block graph of order $n$ different from a complete graph such that $\iota(G) = n/3$, then every
  end-block $B$ of $G$ has at most three vertices.    
\end{cor} 
\begin{proof} 
 By definition, $B$ has exactly one cut vertex and, by Proposition~\ref{propblock1}, at most 2 simplicial vertices. Hence, $B$ has at most 3 vertices.
\end{proof}

\begin{prop} \label{propblock3}
    If $G$ is a block graph of order $n$ different from a path of order 3 such that $\iota(G) = n/3,$ then every cut vertex of $G$ belongs to at most one end-block.
\end{prop}

\begin{proof} 
Suppose that $v$ is a cut vertex of $G$ belonging to at least two end-blocks and let $D$ be a minimum isolating set of $G$ not containing simplicial vertices.
If at least one of the end-blocks containing $v$ has 3 vertices, then $v$ belongs to $D$.
If all end-blocks have 2 vertices, then $v\in N[D]$.
Now, let $u$ be a simplicial vertex belonging to one of the smallest end-blocks containing $v$.
Then, in both cases, a set with no simplicial vertex is isolating in $G$ if and only if is isolating in $G-u$.
Therefore, by Proposition~\ref{propnosimplicial}, we have $\iota(G)=\iota(G-u) \leq \frac{n-1}{3}<\frac{n}{3}$, a contradiction.
\end{proof}


\begin{lemma}\label{rem:remove} Let $G$ be a graph of order $n$ and let $r\ge 1$. If $D$ is a minimum isolating set of $G-\{u_1,\dots, u_r\}$ such that $D$ is also an isolating set of $G$, then $\iota (G)<\frac n3$.
\end{lemma}
\begin{proof}
If $D$ is a minimum isolating set of $G-\{u_1,\dots, u_r\}$ and an isolating set of $G$, then
$\iota (G)\le \vert D\vert =\iota (G-\{u_1,\dots, u_r\})\le \frac{n-r}3< \frac n3.$
\end{proof}

\begin{theorem} If $G$ is a block graph, then  $\iota (G)=\frac{n}{3}$ if and only if $G \in \mathcal{B}.$
\end{theorem}
\begin{proof} 
If $G\in \mathcal{B}$ is a block graph of order $n$, then $\iota (G)=n/3$ by Proposition~\ref{prop:tipo}.

Now let $G$ be a block graph of order $n$ such that $\iota (G)=n/3$. Hence, $n=3k$, for some integer $k\ge 1$. We use induction on $n$ to prove that $G\in \mathcal{B}$.
If $n=3$, then $G$ is a cycle or a path of order 3, and both graphs belong to $\mathcal{B}$. 
Now suppose that $n=6$. If there is a vertex $u$ of degree at least $4$, then $|N[u]|\ge 5$ and $\{u\}$ is an isolating set. Hence, $\iota (G)<2$.
For $\Delta(G)\le 3$, there are 9 non-isomorphic block-graphs and only three of them have isolation number equal to $2$, and these graphs belong to  $\mathcal{B}$, because they can be obtained by attaching $P_3$ or $C_3$ at the vertices of the complete graph $K_2$  (see Figure~\ref{fig:blockOrder6}).
\begin{figure}[ht!]
			\centering
			\includegraphics[width=0.8\linewidth]{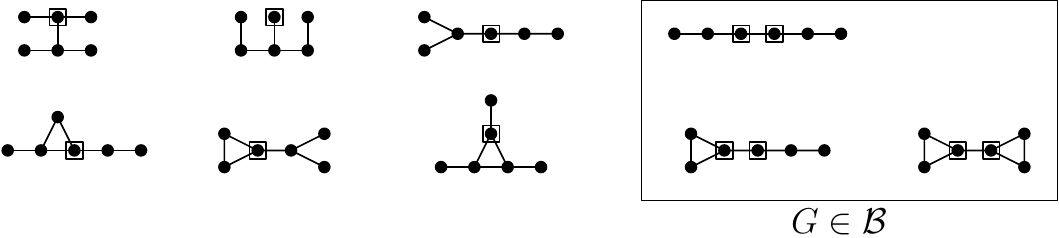}
			\caption{There are 9 block-graphs of order 6 with $\Delta(G)\le 3$. Three of them have isolation number equal to 2 and they all belong to $\mathcal{B}$. Squared vertices form a minimum isolating set in all cases. }
			\label{fig:blockOrder6}
\end{figure}
		
Now suppose that $G$ is a block graph of order $n=3k\ge 9$ such that $\iota (G)=n/3$.
Let $P=(v_0, v_1, \ldots, v_d)$ be a diametral path in $G$ and let $B_i$ be the block of $G$ containing $v_i$ and $v_{i+1}$, $i=0,\ldots,d-1$.
 Since $G$ is a block graph, $v_i$ and $v_{i+1}$ are the only vertices of $P$ belonging to $B_i$.
The choice of $P$ implies also that $v_1,\ldots,v_{d-1}$ are cut vertices, while $v_0$ and $v_d$ are simplicial vertices belonging to the end-blocks $B_0$ and $B_{d-1}$, respectively. Besides, by Proposition~\ref{propblock3}, $v_1$ belongs to exactly two blocks, $B_0$ and $B_1$.
Moreover, by Proposition~\ref{propblock1}, $B_1,\dots,B_{d-2}$ have at most 2 simplicial vertices and, 
 by Corollary~\ref{propblock2}, $B_0$ is $K_3$ or $K_2$.
\smallskip

 {\bf Case 1: $B_1\not= K_2$}. Notice that by Proposition~\ref{propblock3} and because of the choice of $P$, the cut-vertices of $B_1$ different from  $v_1$ and $v_2$ belong to exactly two blocks, $B_1$ and an end-block.
 \begin{itemize}
     \item If $B_1$ has a simplicial vertex $w$, consider a minimum isolating set $D$ of $G-w$ with no simplicial vertices, that exists by Proposition~\ref{propnosimplicial}. 
     Since $v_0$ is a simplicial vertex of the end-block $B_0$, we derive $v_1\in N[D]$. Thus, by Proposition~\ref{propblock3}, $D$ has a vertex of $B_1$.
     Therefore, $D$ is also an isolating set of $G$ and, by Lemma~\ref{rem:remove}, $\iota (G)<n/3$, a contradiction.
     
     \item If $B_1$ has a cut-vertex $w'$ different from $v_1$ and $v_2$, 
     such that the end-block containing $w'$ is $K_2$, let $w$ be the other vertex of this end-block. 
     Consider a minimum isolating set $D$ of $G-w$ with no simplicial vertices, that exists by Proposition~\ref{propnosimplicial},
     and arguing similarly as in the preceding case, we get a contradiction.
     
     \item If all the vertices of $B_1$ different from $v_1$ and $v_2$ are cut-vertices belonging to an end-block $K_3$, we may assume that $B_0$ is also $K_3$, otherwise, $B_0$ must be $K_2$ and the preceding case applies by considering a diametral path beginning at one of the two simplicial vertices of these end-blocks isomorphic to $K_3$.
     Consider the block graph $G'=G-N[v_0]$. 
     If $D$ is an isolating set of $G'$, then $D\cup \{v_1\}$ is an isolating set of $G$. Hence, $\iota (G')=\frac{n-3}3$, because otherwise $\iota (G) \le \iota (G')+1<\frac{n-3}3+1=\frac n3$, a contradiction. Thus, by the inductive hypothesis, $G'\in \mathcal{B}$. Since the vertices of $B_1$ different from $v_1$ and $v_2$ have degree at least $3$ in $G'$, they are in $A(G')$ by Remark~\ref{rem:deg3in block} and are adjacent to $v_2$.  Hence, also $v_2 \in A(G')$. Since $v_1$ is adjacent to any other vertex of the block $B_1$, all belonging to $A(G')$, we deduce that $G\in \mathcal{B}$.
 \end{itemize}

 {\bf Case 2: $B_1= K_2$ and $B_0=K_3$} (See Figure~\ref{fig:case2}a).

 \begin{figure}
			\centering
			\includegraphics[width=0.8\linewidth]{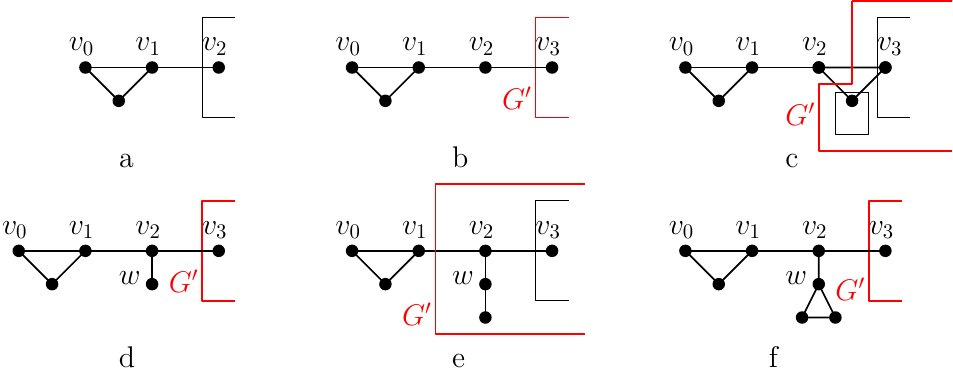}
			\caption{Case $B_0=K_3$ and $B_1=K_2$.}
			\label{fig:case2}
	\end{figure}

 \begin{itemize}
     \item If $\deg_G(v_2)\ge 4$, consider the graph $G'=G-N[v_0]$. If $D$ is an isolating set of $G'$, then $D\cup \{v_1\}$ is an isolating set of $G$. 
     Hence, $\iota (G')=\frac{n-3}3$, because otherwise $\iota (G) \le \iota (G')+1<\frac{n-3}3+1=\frac n3$, a contradiction. 
     Since $G'$ is also a block graph of order at least 6, by the inductive hypothesis $G'\in \mathcal{B}$.   
     Since $\deg_G(v_2)\ge 4$, we derive $v_2\in A(G')$. Hence, $G\in \mathcal{B}$.
     
     \item If $\deg_G(v_2)=2$, consider the graph $G'=G-N[v_1]$ (see Figure~\ref{fig:case2}b). If $D$ is a minimum isolating set of $G'$, then $D\cup \{v_1\}$ is an isolating set of $G$. Thus, since $G'$ is a connected graph of order at least $5$ different from a cycle of order 5, by Theorem~\ref{thm:caro}
     $\iota(G)\le \iota(G')+1\le 1+\frac{n-4}3<\frac n3$, a contradiction.
     
     \item If $\deg_G(v_2)=3$ and $v_2$ belongs to exactly 2 blocks, $B_1$ and $B_2$, then $B_2$ is $K_3$ and $G'=G-N[v_1]$ is a connected graph  (see Figure~\ref{fig:case2}c). 
     If $D$ is a minimum isolating set of $G'$, then $D\cup \{v_1\}$ is an isolating set of $G$. Thus,
     $\iota(G)\le \iota(G')+1\le 1+\frac{n-4}3<\frac n3$, a contradiction. 
     
     \item If $\deg_G(v_2)=3$ and $v_2$ belongs to 3 blocks, $B_1$, $B_2$ and $B'$, then the 3  blocks  must be $K_2$. Let $w$ be the vertex of $B'$ different from $v_2$.
     
     If $B'$ is an end-block, then consider $G'=G-(N[v_1]\cup \{w\})$ (see Figure~\ref{fig:case2}d). If $D$ is a minimum isolating set of $G'$, then $D\cup \{v_1\}$ is also an isolating set of $G$, so that 
     $\iota(G)\le \iota(G')+1\le 1+\frac{n-5}3<\frac n3$, a contradiction. 

     If $B'$ is not an end-block, then $w$ is a cut-vertex belonging to exactly one end-block $B''$, because of the choice of $P$. 

     $-$ If $B''=K_2$  (see Figure~\ref{fig:case2}e), then $G'=G-N[v_0]$ is a block graph of order at least 6. Moreover, if $D$ is a minimum isolating set of $G'$, then $D\cup\{ v_1 \}$ is a minimum isolating set of $G$. Hence, $\iota (G')=\frac{n-3}3$, because otherwise
     $\iota (G)\le \iota(G')+1<\frac{n-3}3+1=\frac n3$, a contradiction.
     By the inductive hypothesis, $G'\in \mathcal{B}$. Since $v_2\in A(G_2)$, we have $G\in \mathcal{B}$.

     $-$ If $B''=K_3$  (see Figure~\ref{fig:case2}f), let $G'=G-(N[v_1]\cup N[w])$. If $D$ is an isolating set of $G'$, then $D\cup \{v_1,w\}$ is also an isolating set of $G$.
     Since $G'$ is connected, if $G'$ has order at least 3, then
     $\iota (G)\le \iota (G')+2\le \frac{n-7}3 +2<\frac n3$, a contradiction.
     If $G'$ has order 1, then $G$ has order 8 and 
     $\iota (G)=2<\frac 83$, a contradiction.
     If $G'$ has order 2, then $G\in \mathcal{B}$.
\end{itemize}
 \medskip

 {\bf Case 3: $B_1= K_2$ and $B_0=K_2$}.

  \begin{figure}[t]
			\centering
			\includegraphics[width=0.6\linewidth]{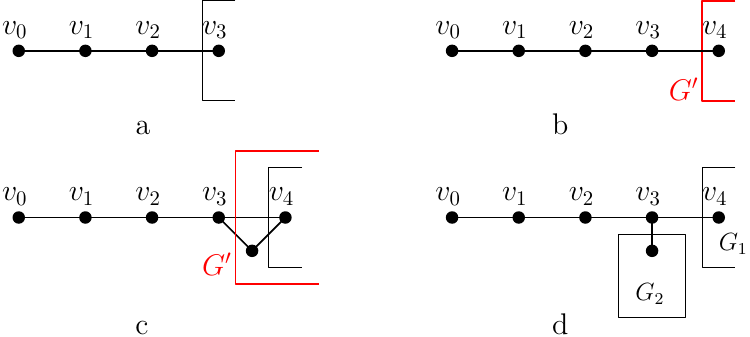}
			\caption{Case $B_0=B_1=K_2$ and $\deg(v_2)=2$.}
			\label{fig:case3}
	\end{figure}
		
 \begin{itemize}
     \item If $\deg_G(v_2)=2$ (see Figure~\ref{fig:case3}a),  we distinguish cases according to the degree of $v_3$.

     - If $\deg _G(v_3)\ge 4$, then consider the block graph $G'=G-N[v_1]$. If $D$ is a minimum isolating set of $G'$, then $D\cup \{v_1\}$ is an isolating set of $G$. Hence, $G'$ is a block graph such that $\iota (G')=\frac{n-3}3$, because otherwise $\iota (G)\le \iota (G')+1<\frac{n-3}3+1=\frac n3$, a contradiction. Since $G'$ has order at least 6, by the inductive hypothesis, $G'\in \mathcal{B}$. Since $\deg_{G'}(v_3)\ge 3$, we have $v_3\in A(G')$ by Remark~\ref{rem:deg3in block}. Hence, $G\in \mathcal{B}$.

     - If $\deg _G(v_3)=2$, then $v_3$ belongs to exactly 2 blocks, $B_2$ and $B_3$ (see Figure~\ref{fig:case3}b). Then, the graph $G'=G-\{v_0,v_1,v_2,v_3\}$ is connected and has order at least 5. If $D$ is a minimum isolating set of $G'$, then $D\cup \{v_2\}$ is an isolating set of $G$.
     Hence, $\iota (G)\le \iota (G')+1\le \frac{n-4}3+1<\frac n3$, a contradiction.
     
     - If $\deg _G(v_3)=3$  and $v_3$ belongs to exactly 2 blocks (see Figure~\ref{fig:case3}c), then the reasoning of the preceding case also applies.

     - If $\deg _G(v_3)=3$ and $v_3$ belongs to 3 blocks, $B_2$, $B_3$ and a block $B'$, then these 3 blocks must be $K_2$.
     Let $G_1$ and $G_2$ be the components of $G-v_3$ such that $G_1$ contains $v_4$ and $G_2$ contains no vertex of the path $P$  (see Figure~\ref{fig:case3}d). Let $n_1$ and $n_2$ be the order of $G_1$ and $G_2$, respectively.

     If $n_2=1$, consider a minimum isolating set $D$ of $G_1$. Then,
     $D\cup \{v_2\}$ is an isolating set of $G$ and 
     $\iota (G)\le \iota (G_1)+1\le \frac{n-5}3+1<\frac n3$, a contradiction.

      If $n_2\ge 3$, then $n_1\ge 2$, because of the choice of $P$.
      If $n_1\ge 3$, then consider a minimum isolating set $D_1$ of $G_1$ and a minimum isolating set $D_2$ of $G_2$. Then,
     $D_1\cup D_2\cup  \{v_2\}$ is an isolating set of $G$.
     Hence, 
     $\iota (G)\le \iota (G_1)+\iota(G_2)+1\le \frac{n_1}3+\frac {n_2}3+1=
     \frac{n-4}3+1<\frac n3$, a contradiction.
     If $n_1=2$, then consider the block graph $G'=G-N[v_1]$. Then, $\iota (G')=\frac{n-3}3$, since otherwise a minimum isolating set of $G'$ together with the vertex $v_1$ is an isolating set of $G$ and we would have 
     $\iota (G)\le \iota (G')+1<\frac {n-3}3+1=\frac n3$. Since $n-3\ge 6$, by the inductive hypothesis $G'\in \mathcal{B}$. Since $v_3$ is at distance 2 from a leaf in $G'$, then $v_3\in A(G')$ and, consequently, $G\in \mathcal{B}$.

     If $n_2=2$, then $n_1\ge 3$. We proceed as in the case $n_1=2$ and $n_2\ge 3$, by interchanging the role of $G_1$ and $G_2$, and derive that $G\in \mathcal{B}$.

  \begin{figure}[t]
			\centering
			\includegraphics[width=0.85\linewidth]{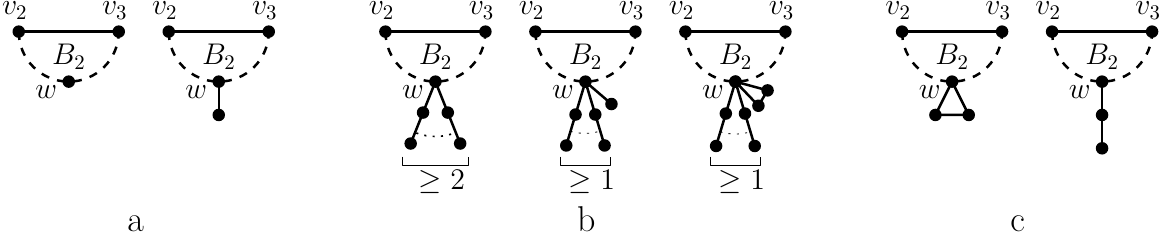}
			\caption{Case $B_0=B_1=K_2$, $\deg(v_2)\ge 3$ and $v_2$ belongs to $2$ blocks. The vertex $w$ is in $B_2$ and different from $v_2$ and $v_3$}
			\label{fig:case3degv3ge3b2}
	\end{figure}
 
    \item If $\deg_G(v_2)\ge 3$ and $v_2$ belongs to exactly $2$ blocks, concretely $B_1$ and $B_2$, then $B_2$ has at least one vertex different from $v_2$ and $v_3$. Let $w$ be such a vertex. Then,  $w$ is a simplicial vertex or $w$ belongs to a block $B'$, different from $B_2$.
    
    If $B'$ is an end-block, then $B'$ is $K_2$ or $K_3$, by Proposition~\ref{propblock3}, and $w$ belongs to exactly 2 blocks, $B_2$ and $B'$.
    
If $B'$ is not an end-block, then we can assume that $B'=K_2$, since otherwise Case 1 applies by considering a suitable diametral path through $w$, because of the choice of $P$. Moreover, by the choice of $P$ and by Proposition~\ref{propblock3}, if $B'=K_2$ is not an end-block, then the vertex of $B'$ different from $w$ belongs to exactly one end-block that must be $K_2$ or $K_3$, but Case 2 applies for a suitable diametral path that goes through $w$ if this end-block is $K_3$.
Hence, if $B'$ is not an end-block, then the vertex $w'$ of $B'$ different from $w$ has degree 2 and is adjacent to a leaf. Let $\mathcal{H}_w$ denote  the set of all maximal connected subgraphs of $G$ containing $w$ but not as a cut vertex and without vertices from $P$, whenever $w$ is not a simplicial vertex. If $\vert \mathcal{H}_w \vert=1$, then $\mathcal{H}_w$ contains either $K_2$, $K_3$ or $P_3$. If $\vert \mathcal{H}_w \vert=2$, then $\mathcal{H}_w$  contains at least one subgraph  isomorphic to $P_3$ and at most one  subgraph isomorphic to $K_2$ or $K_3$. 

Hence, it is enough to analyse the following cases for the vertices $w$ in $B_2$ different from $v_2$ and $v_3$.

If at least one of the vertices $w$ is  simplicial  or  belongs to exactly two blocks, $B_2$ and $B'$, where $B'$ is an end-block isomorphic to $K_2$ (see Figure~\ref{fig:case3degv3ge3b2}a), consider the graph $G'$ obtained by removing only $w$, if $w$ is simplicial, and by removing the two vertices of $B'$, otherwise. Notice that $G'$ is a connected graph and there is a minimum isolating set  $D$ of $G'$ that contains $v_2$, because $\deg_{G'}(v_1)=2$. Hence, $D$ is also an isolating set of $G$. Therefore, 
$\iota(G)\le \iota (G')\le \frac{\vert V(G') \vert}3<\frac n3$, a contradiction.

If at least one of the vertices $w$ is a cut vertex such that $\mathcal{H}_w$ has at least 2 elements (see Figure~\ref{fig:case3degv3ge3b2}b), then consider the graph $G'$ obtained by removing all the vertices belonging to some subgraph of $\mathcal{H}_w$. Notice that at least 4 vertices are removed.
Let $D$ be a minimum isolating set of $G'$. Then, $D\cup \{w\}$ is an isolating set of $G$.
Therefore, since $G'$ is connected, we have
$\iota(G)\le \iota (G')+1\le \frac{\vert V(G') \vert}3+1<\frac {n-4}3+1<\frac n3$, a contradiction. 

Finally, if the preceding cases do not apply, it means that for every vertex $w$ either $\mathcal{H}_w=\{K_3\}$ or  $\mathcal{H}_w=\{P_3\}$ (see Figure~\ref{fig:case3degv3ge3b2}c).  Consider the block graph $G'=G-N[v_1]$. A set formed by a minimum isolating set of $G'$ together with the vertex $v_2$ is an isolating set of $G$. Hence, $\iota(G')=\frac{n-3}3$, because otherwise
$\iota (G)\le \iota (G')+1<\frac{n-3}3 +1\le \frac n3$. By the inductive hypothesis, $G'\in \mathcal{B}$ and $v_3\in A(G')$, because every vertex $w$ of $B_2$ different from $v_2$ and $v_3$  belongs to $A(G')$. Hence, $G\in \mathcal{B}$.
     
     \item If $\deg_G(v_2)\ge 3$ and $v_2$ belongs at least $3$ blocks: $B_1$, $B_2$ and at least one more block $B'$. 
     If $B'$ is an end-block, then $B'=K_2$ or $B'=K_3$, by Proposition~\ref{propblock1}. 
     If $B'$ is not an end-block, then we can assume that $B'=K_2$ because of the choice of $P$, since Case 1 applies by considering a suitable diametral path. Moreover, by the choice of $P$ and by Proposition~\ref{propblock3}, if $B'=K_2$ is not an end-block, then the vertex of $B'$ different from $v_2$ belongs to an end-block that must be $K_2$ or $K_3$, and in this last case, Case 2 applies for a suitable diametral path.
     
     Hence, it only remains to consider the following cases: 
     either $B'=K_2$ is an end-block (see Figure~\ref{fig:case3degv3ge3b3}a);
     or $B'=K_3$ is an end-block (see Figure~\ref{fig:case3degv3ge3b3}b); 
     or $B'=K_2$ is not an end-block, and the vertex $w$ of $B'$ different from $v_2$ belongs to an end-block $B''$ that is $K_2$ (see Figure~\ref{fig:case3degv3ge3b3}c). 
  \begin{figure}[t]
			\centering
			\includegraphics[width=0.7\linewidth]{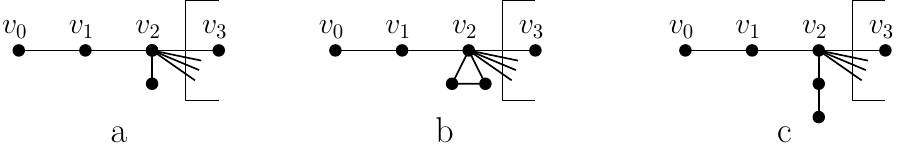}
			\caption{Case $B_0=B_1=K_2$, $\deg(v_2)\ge 3$ and $v_2$ belongs to at least $3$ blocks.}
			\label{fig:case3degv3ge3b3}
	\end{figure}
 
     Consider the graph $G'$ obtained by removing all the vertices of the blocks $B'$ and $B''$, if this last block exists, different from $v_2$.
     Since $\deg_{G'}(v_1)=2$, there is a minimum isolating set $D$ of $G'$ containing $v_2$ and, in all cases, $D$ is also an isolating set of $G$.
     Therefore, since $G'$ is a connected graph of order at least 3, we have $\iota(G)\le \iota (G')\le \frac{\vert V(G')\vert}3<\frac n3$, a contradiction.
\end{itemize}

\end{proof}

\section{Concluding remarks}\label{sec:concluding}
Our goal is to characterize all the graphs  attaining the upper bound on the isolation number.
Notice that, in some sense,  the obtained results have  the same flavour as the characterization of graphs attaining the upper bound on the domination number.
It is well known that a graph $G$ of order $n$ without isolated vertices has domination number at most $n/2$ \cite{ore}, and the graphs attaining this bound are the cycle $C_4$ and the corona graphs, $G\circ K_1$, obtained  by hanging a leaf to every vertex of a graph $G$ \cite{Fink,payan},
i.e., by attaching a copy of $K_2$ to every vertex of $G$ by identifying the vertex of $G$ with a vertex of $K_2$.

By Proposition~\ref{prop:tipo},  every graph in the family $\mathcal{G}$, that is, obtained by properly attaching a copy of $P_3$, $C_3$, $H_6^1$,  $H_6^{2a}$, $H_6^{2b}$ and $H_6^{3}$ to every vertex of a graph $G_0$, provides a graph attaining the upper bound on the isolation number.
For trees and block graphs, the graphs attaining the upper bound are precisely those belonging to $\mathcal{G}$. 
For unicyclic graphs, the only graphs not in $\mathcal{G}$  attaining the upper bound are the cycles $C_6$ and $C_9$. 

For connected graphs of order 6, it can be checked that there are 3 non-isomorphic graphs with isolation number equal to 2 not in $\mathcal{G}$ (see Figure~\ref{fig:notinGorder6}). For connected  graphs of order 9 there are several graphs not in $\mathcal{G}$ with isolation number equal to 3.

Our feeling is that there are only a few graphs attaining the upper bound on the isolation number not belonging to $\mathcal{G}$. It is open problem to determine them all.

 \begin{figure}[t]
			\centering
			\includegraphics[width=0.4\linewidth]{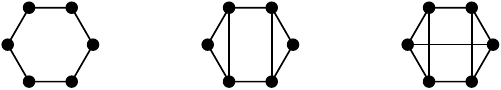}
			\caption{Connected graphs of order $6$ and isolation number 2 not in $\mathcal{G}$.}
			\label{fig:notinGorder6}
	\end{figure}

\bibliography{bib file}

\begin{thebibliography}{99}
\bibitem{alvarado15} J. Alvarado, S. Dantas, D. Rautenbach, Distance k-domination, distance k-guarding, and distance k-vertex cover of maximal outerplanar graphs,
{\it Discrete Appl. Math.} 194 (2015) 154–159.
\bibitem{borg20} P. Borg, K. Fenech, P. Kaemawichanurat, Isolation of $k$-cliques, {\it Discrete Math.} 343 (2020) 111879.
\bibitem{borg22} P. Borg, K. Fenech, P. Kaemawichanurat, Isolation of $k$-cliques II, {\it Discrete Math.} 345 (2022) 112641.
\bibitem{canales15} S. Canales, G. Hernández, M. Martins, I. Matos, Distance domination, guarding and covering of maximal outerplanar graphs, {\it Discrete Appl. Math.}
181 (2015) 41–49.
\bibitem{adriana} Y. Caro, A. Hansberg, Partial domination - the isolation number of a graph, {\it Filomat}  31 (12) (2017) 3925--3944. 
\bibitem{my} A. Dapena, M.  Lema\'nska, M.J. Souto-Salorio, F.J.  Vazquez-Araujo,  Isolation number versus domination number of trees. {\it Mathematics}  9 (12) (2021) 1325.
\bibitem{Fink} J.M. Fink, M.S. Jacobson, L.F. Kinch, J. Roberts, On graphs having domination number half their order, {\it Period. Math. Hungar.} 16 (1985) 287-293.
\bibitem{haynes} T.W. Haynes, S.T. Hedetniemi, P.J. Slater, {\it Fundamentals of domination in graphs},
Marcel Dekker, New York, 1998.
\bibitem{LYD} X. Li, G. Yu, K.C. Das, The Average Eccentricity of Block Graphs: A Block Order Sequence Perspective,  {\it Axioms}  11(3) (2022) 114. 
\bibitem{ore} O. Ore. {\it Theory of graphs},  Amer. Math. Soc. Colloq. Publ., 38, 1962.
\bibitem{payan} C. Payan, N. Xuong, Domination-balanced graphs, {\it J. Graph Theory} 6 (1982) 23-32. 
\end{thebibliography}

\end{document}